\documentclass[11pt]{article}
\usepackage{amsmath, amssymb, amsthm, amsfonts,graphics, graphicx,algorithm2e}
\usepackage{subfigure}
\providecommand{\keywords}[1]{\textbf{\textit{Keywords---}} #1}
\usepackage[latin1]{inputenc}
\usepackage{latexsym}
\usepackage{graphics,epsfig}
\usepackage{amsfonts}
\usepackage[english]{babel}
\usepackage{verbatim}
\SetKwInput{KwInput}{{\bf Input}}
\SetKwInput{KwOutput}{{\bf Output}}

\usepackage{enumitem}

\usepackage[hidelinks]{hyperref}
\hypersetup{
  colorlinks   = true, %Colours links instead of ugly boxes
  urlcolor     = blue, %Colour for external hyperlinks
  linkcolor    = blue, %Colour of internal links
  citecolor   = blue  %Colour of citations
}

\newtheorem{theorem}{Theorem}
\newtheorem{lemma}[theorem]{Lemma}
\newtheorem{proposition}[theorem]{Proposition}
\newtheorem{corollary}[theorem]{Corollary}%[section]
\newtheorem*{conjecture}{Conjecture}

\theoremstyle{definition}
\newtheorem{definition}{Definition}

\makeatletter
\g@addto@macro\th@remark{\thm@headpunct{:}}
\makeatother
\theoremstyle{remark}
\newtheorem*{remark}{Remark}
\newtheorem*{remarks}{Remarks}
\newcommand{\CFG}{\text{CFG}}

\begin{document}
\title{A maximizing characteristic for critical configurations of chip-firing
games on digraphs}

\author{
  Hoang Thach NGUYEN\footnote{Institute of Mathematics, 18 Hoang Quoc Viet
    street, Cau Giay district, Hanoi, Vietnam, Email:
    \textsf{nhthach@math.ac.vn}}, 
    \and Thi Thu Huong TRAN\footnote{Vietnamese-German University, Le Lai street,
      Hoa Phu ward, Thu Dau Mot City, Binh Duong, Vietnam, Email:
      \textsf{huong.ttt@vgu.edu.vn and ttthuong@math.ac.vn}}
    }

\maketitle

\begin{abstract}
Aval \emph{et al.} proved in \cite{ADDL16} that starting from a critical
configuration of a chip-firing game on an undirected graph, one can never
achieve a stable configuration by reverse firing any non-empty subsets of its
vertices. In this paper, we generalize the result to digraphs with a global sink
where reverse firing subsets of vertices is replaced with reverse firing
multi-subsets of vertices. Consequently, a combinatorial proof for the duality
between critical configurations and superstable configurations on digraphs is
given. Finally, by introducing the concept of energy vector assigned to each
configuration, we show that critical and superstable configurations are the
unique ones with the greatest and smallest (w.r.t. the containment order),
respectively, energy vectors in each of their equivalence classes.   
\end{abstract}


\bigskip
\section{Introduction}\label{sec:intro} 
Originally developped in the early nineties of the last centuries in the context
of self-organized criticality~\cite{Dha90} and as a ``balancing game'' on
graphs~\cite{BL92,BLS91}, chip-firing games (CFG) have become an
attractive mathematical model in combinatorics~\cite{Big99,Mer05,Mer97,PP16}.
In recent years, new connections have been found between CFG and the Riemann-Roch
theory on graphs~\cite{AB11,BN07} and potential theory on graphs~\cite{BS13,GK15}.
This manuscript contributes some new results to the latter category and at the
same time provides a combinatorial insight on existing results.

\smallskip

We consider a CFG on a digraph $G$ with a global sink. A configuration is a
distribution of chips on the non-sink vertices. A transition, called a firing,
consists in choosing a vertex and sending
one of its chips along each out-going arc to its neighbors. A firing is legal if
the chosen vertex has at least as many chips as its out-going arcs. Suppose we
restrict the game to non-negative configurations and legal firings. Then, a
stable configuration is a configuration in which all the non-sink vertices
cannot fire. A critical configuration is a stable configuration which is
``attainable'' from a configuration with many chips. The detailed descriptions
of the model will be given in Section \ref{sec:prelim}. Further discussions on
critical configurations can be found in \cite{HLMPPW08}. 

In order to investigate the critical configurations, we introduce the notion of
$G$-strongly positive scripts. The $G$-strongly positive scripts are closely
related to Speer's algorithm~\cite{Spe93} (indeed, the minimum $G$-strongly
positive script coincides with the output of Speer's algorithm) and are
equivalent to Perkinson {\em et al.}'s ``burning configurations''~\cite{PPW11}.
The key observation is that reverse-firing a $G$-strongly positive script
``increases'' the configurations. Thus, by repeatedly reverse-firing a
$G$-strongly positive script then stabilizing, we obtain an ``increasing''
sequence of stable configurations. Since there are only finitely many stable
configurations in each equivalence class, this process must eventually stop at
a fixed point, which is the critical configuration.

\smallskip

Our contribution in this paper is threefold. First, we show that for CFG on
digraphs with a global sink, among all the stable configurations of each
equivalence class, the critical configuration has the maximum weight
(Theorem \ref{T:max}). This has
been proved in \cite{PP16} for Eulerian digraphs, but the question remained open
for digraphs with a global sink. Secondly, we give an
extension of Aval {\em et al.}~\cite{ADDL16}'s result from undirected graphs to
digraphs with a global sink. Namely, one cannot obtain a stable configuration
from a critical one by reverse-firing a non-empty multi-subset of
vertices (Theorem \ref{T:script}). Using this result, we revisit the duality
between critical configurations and superstable configurations (Theorem
\ref{T:dua}). It should be noted that the duality is well-known in the
literature: the result for Eulerian digraphs was given by Holroyd \emph{et
al.}~\cite{HLMPPW08}, for strongly connected digraphs by Asadi and
Backman~\cite{AB11}, for digraphs with a global sink by Perkinson \emph{et
al.}~\cite{PPW11}. There are remarkable differences between the techniques used
in those proofs. For the case of digraphs with a global sink, Perkinson
{\em et al.} used advanced algebra techniques such as the coordinate ring and
Gr\"obner bases, while Asadi and Backman's proof is a combinatorial one. Our
proof is combinatorial and is independent from Asadi and Backman's proof. A
notable difference is that the attainability in their proof requires that the
firing/reverse-firing sequences to be legal, while the legality restriction is
not imposed in our proof. Lastly, we give an energy maximizing (resp.
minimizing) characteristic for critical (resp. superstable) configurations
(Theorem \ref{T:max2} and Corollary \ref{C:sma}). Unlike previous studies on CFG
and potential theory~\cite{BS13, GK15} where energies are defined as norms, in this
paper, the energies are defined as vectors and are compared using the containment
order. This order is not unfamiliar in the CFG literature: its restriction on
the set of accessible configurations (given an initial configuration) coincides
with the accessibility order, which has been extensively investigated in
connection with lattice structures~\cite{CPT13, LP01}.

\smallskip

The manuscript is organized as follows. Section \ref{sec:prelim} gives a brief
summary of relevant definitions and fundamental results on $\CFG$ and critical
configurations. Section \ref{sec:critical} presents some properties and the
maximizing characterization of critical configurations. Section
\ref{sec:duality} discusses the duality and the minimizing characterization of
superstable configurations.

\bigskip

\section{Preliminaries}\label{sec:prelim}
Unless stated otherwise, the vectors in this paper are row vectors. Let $a$ be a
vector in $\mathbb Z^n$. We denote by
$a_i$ the $i$-th component of $a$. We say that a vector is
non-negative (resp. positive) if all of its components
are non-negative (resp. positive). The zero vector is denoted by
$\pmb 0$ and the all-one vector by $\pmb 1$.

The {\em support} of $a$ is defined by $\mbox{supp}(a) = \{i\ \vert\ a_i
\neq 0\}$. The {\em weight} of $a$ is $w(a) = \sum_{i = 1}^n a_i$.

For a matrix $M$, let $M_i$ denote the $i$-th row vector of $M$ and $M_{ij}$
the entry at row $i$ and column $j$ of $M$.

The {\em containment order} on $\mathbb Z^n$ is defined by: $a \succeq b \iff
a_i \ge b_i$ for all $i=1,2,\cdots,n$.

\medskip

Let us now give some basic notions of chip-firing games on digraphs with a
global sink.

Let $G = (V, E)$ be a directed multigraph without loops, with $n + 1$ vertices,
{\em i.e.} $V = \{1, 2, \cdots, n + 1\}$. 
The \emph{out-going degree} (resp. \emph{in-going
degree}) of a vertex $i$ is denoted by $d^+_i$ (resp. $d^-_i$). The number of
edges going from $i$ to $j$ is denoted by $e_{i, j}$.

A {\em source component} (resp. {\em sink component}) of $G$ is a strongly
connected component without in-going (resp. out-going) edge from other strongly
connected components.
A vertex $s$ of $V$ is a {\em sink} if $\{s\}$ is a sink component. It is a {\em global
sink} if it is a sink and for every other vertex $i$, there exists a
directed path from $i$ to $s$. Note that a global sink is unique if it exists.
In the rest of the paper, we suppose that $G$ has a global sink at the vertex
$n + 1$.

\smallskip

The graph $G$ can be represented by its {\em Laplacian matrix} $\widetilde \Delta \in
\mathbb Z^{(n + 1) \times (n + 1)}$ whose entries are defined as follows:
\begin{equation*}
  \widetilde{\Delta}_{ij} =
  \begin{cases}
    d^+_i \quad & \text { if } i = j\,,\\
    -e_{i,j} \quad & \text{ if } i \ne j\,.
  \end{cases}
\end{equation*}
The {\em reduced
Laplacian matrix} $\Delta$ is the matrix obtained from $\widetilde \Delta$ by removing the row and the
column corresponding to the sink. 

The Laplacian and reduced Laplacian matrices play an important role in the study
of combinatorial properties of graphs and of chip-firing games on graphs
(defined hereafter) as well, see 
\cite{BL92, BLS91} and the references therein. We recall here two useful
properties for the Laplacian and the reduced Laplacian of digraphs with a global sink.

\begin{enumerate}[label=\roman*)]
  \item The kernel of $\widetilde{\Delta}$ is spanned by a non-negative vector
    $v$ such that $v_{n + 1} > 0$. Consequently, the rank of
    $\widetilde{\Delta}$ is $n$.
  \item The matrix $\Delta$ is non-singular. Moreover, all the entries of
    $\Delta^{-1}$ are
    non-negative.
\end{enumerate}

The first property is a particular case of \cite[Proposition 3.1]{BL92}. The
non-singularity is a corollary of the first property. The non-negativity of
$\Delta^{-1}$ is a property of non-singular M-matrix (see, for instance,
\cite{plemmons77, PS04}).

\medskip

A {\em chip-firing game} ($\CFG$) on $G$, denoted by $\CFG(G)$, is a discrete dynamical system in which:
\begin{itemize}
  \item[-] A {\em chip configuration} ({\em configuration} for short) is a vector
    in $\mathbb Z^n$ whose coordinates represent the
    numbers of chips at non-sink vertices \footnote{Typically, the
      configurations must be
      non-negative. In this paper, we allow the coordinates of a configuration
    to be negative.}. We will use bold letters such as
    $\pmb a$  to denote configuration vectors.
  \item[-] The transitions occur according to the {\em firing rule}, defined as
    follows. A vertex is {\em active} if it has at least as many chips as its
    out-going degree. An active vertex can {\em fire} by sending one chip along
    each of its out-going edge to its neighbors. That is, if $\pmb b$ is the configuration
    obtained from $\pmb a$ by firing an active vertex $i$, we write $\pmb a
    \xrightarrow{i} \pmb b$, then:
    \[
      b_j = 
      \begin{cases}
        a_i - d^+_i & \mbox{ if } j = i\,,\\
        a_j + e_{i,j} & \mbox{ otherwise,}
      \end{cases}
    \]
    or equivalently:
    \[
      \pmb b = \pmb a - {\Delta}_i\,.
    \]
\end{itemize}

\smallskip

A (unconstrained) {\em firing sequence}  is a sequence of non-sink vertices $(s_1, s_2, \cdots, s_k)$ in $V$. The firing sequence $(s_1, s_2, \cdots, s_k)$ is called 
{\em legal} from $\pmb a$ if there exists a sequence of configurations $\pmb a = \pmb a_0, \pmb a_1, \cdots, \pmb
a_k$ such that $\pmb a_{r - 1} \xrightarrow{s_{r}} \pmb a_{r}$ for all $1 \le r
\le k$ and $s_r$ is active in $\pmb a_{r - 1}$.

The {\em firing script} of a legal firing sequence
$(s_1, s_2, \cdots, s_k)$ is a vector $\sigma \in \mathbb N^n$, where $\sigma_i$
is the number of occurrences of the vertex $i$ in the sequence $(s_1,s_2,\cdots,s_k)$.
The firing script provides a direct way to compute the result of a firing
sequence: if $\pmb b$ is obtained from
$\pmb a$ by firing sequentially $s_1, s_2, \cdots, s_k$ and $\sigma$ is the
corresponding firing script,
then 
\begin{equation}
  \pmb b = \pmb a - \sigma \Delta\,.
  \label{eq:firingequation}
\end{equation}
When referring to vectors in $\mathbb N^n$ without mentioning a firing sequence,
we will use the term {\em $n$-scripts}, or simply {\em scripts}.

\medskip

We say that two configurations $\pmb a$ and $\pmb b$ are {\em linearly
equivalent}, denoted by $\pmb a \sim \pmb b$, if there exists a vector
$\sigma \in \mathbb Z^n$ verifying
(\ref{eq:firingequation}). The linear equivalence is an equivalent relation on
$\mathbb Z^{n}$. The quotient group $\mathbb
Z^n/\langle\Delta_1,\dots,\Delta_n\rangle$ is called the {\em Sandpile group} of
$G$ and is denoted by $\mathcal{SP}(G)$. When the context is clear, we refer to elements of $\mathcal{SP}(G)$ simply as equivalent classes.

\smallskip

Note that $\pmb a \sim \pmb b$ does not imply that $\pmb b$ can be obtained from
$\pmb a$ by a legal firing sequence. In fact, even if 
$\pmb b = \pmb a - \sigma \Delta$ 
for some non-negative $\sigma$, it is not always the case that there
exists a legal firing sequence leading from $\pmb a$ to $\pmb b$.

\medskip

A configuration is said to be {\em stable} if it does not have any active
vertex. It is known that for each $\CFG$ with a global sink and for any non-negative
configuration $\pmb a$, there exists a unique stable configuration, denoted by
$\pmb a^\circ$, which can be obtained from $\pmb a$ by a legal firing sequence.
A process of firing 
from $\pmb a$ to $\pmb a^\circ$ is called a {\em
stabilization} of $\pmb a$.
Note that there may be many legal firing sequences from $\pmb a$ to $\pmb a^\circ$, but they
all have the same firing script. In fact, they are the longest legal firing
sequences from $\pmb a$. It is straightforward that if $\pmb a$ and $\pmb b$ are non-negative
configurations, then
\[
  (\pmb a+\pmb b)^\circ = (\pmb a+\pmb b^\circ)^\circ =(\pmb  a^\circ+\pmb
  b^\circ)^\circ\,.
\]
Moreover, if $\pmb b$ is obtained from $\pmb a$ by a legal firing sequence with firing script $\sigma$, then $$(\pmb a-\sigma\Delta)^\circ =\pmb b^\circ.$$
We refer the interested reader to
\cite{BLS91,HLMPPW08,LP01} for more details on these results.

\medskip

We end this section with some important properties of critical
configurations, also known as recurrent configurations. There are several definitions of critical configurations in the literature. In \cite{HLMPPW08}, it is shown that all those definitions are equivalent for
$\CFG$ on a digraph with a global sink. We recall here the definition that we
consider most useful for our purposes in this paper.

Let $\pmb c_{max} = (d_1^+-1, d_2^+-1,\cdots, d_n^+-1)$ be the maximum stable
configuration by the containment order. 
\begin{definition}[Critical configuration]\mbox{}

  A non-negative configuration $\pmb a$ is critical if and only if for
  any configuration $\pmb b$, there exists a non-negative configuration
  $\pmb c$ such that $\pmb a = (\pmb b + \pmb c)^\circ$.
\end{definition}
\smallskip

The following lemma presents some known properties of critical configurations.

\begin{lemma}[\cite{HLMPPW08}]
  Consider a $\CFG$ on a graph $G$ with a global sink.
  \begin{enumerate}[label=\roman*)]
    \item Each equivalent class of $\mathcal{SP}(G)$ has exactly one critical
      configuration.
    \item If $\pmb a$ is critical, $\pmb b$ is stable, and $\pmb b \succeq \pmb
      a$, then $\pmb b$ is critical.
    \item A non-negative configuration $\pmb a$ is {\em critical} if it is stable and
      there exists a non-negative configuration $\pmb c$ such that
      $\pmb a = (\pmb c_{max} + \pmb c)^\circ$.
  \end{enumerate}
  \label{lem:critproperties}
\end{lemma}

It is difficult to decide if a configuration is critical or not using the above
definitions. However, there exist efficient algorithms to recognize
critical configurations, for instance Dhar's \emph{burning algorithm} for
undirected graphs~\cite{Dha90} and Speer's \emph{script algorithm} for connected
digraphs~\cite{Spe93}. In the next section, we give a maximizing
characterization for the critical configurations on their equivalent classes.

\section{Critical configurations as energy vector maximizers}\label{sec:critical}
In this section, we prove that we never reach a stable configuration from a
critical one by reverse-firing unconstrainedly any non-empty multi-subsets of vertices (Theorem
\ref{T:script}). This is an analogue of Aval \emph{et al.}'s
result~\cite{ADDL16} in the case of digraphs with a global sink.
Then, by associating each configuration with an energy vector, we will show that
among stable configurations in a same equivalence class, the critical
configuration is the unique one with the greatest energy vector by the
containment order (Theorem \ref{T:max2}). An affirmative answer for a question
in \cite{PP16} about the maximum weight of critical configurations is also given
using $G$-strongly positive scripts.

Before presenting these results, we first introduce the concepts of $G$-positive
scripts and $G$-strongly positive scripts.

\medskip

\begin{definition}[$G$-positive and $G$-strongly positive
  scripts]\mbox{}

  Let $\sigma$ be an $n$-script. We say that:
  \begin{enumerate}[label=\roman*)]
    \item $\sigma$ is \emph{$G$-positive} if $\sigma\Delta \succ \pmb 0$;
    \item $\sigma$ is \emph{$G$-strongly positive} if it is $G$-positive and
      $supp(\sigma\Delta)$ intersects each source component of $G$.
  \end{enumerate}
\end{definition}

\begin{remarks}\mbox{}
  \begin{itemize}
    \item[-] In case that $G\setminus \{n+1\}$ is strongly connected, the
      concepts of $G$-positiveness and $G$-strong positiveness coincide. 
    \item[-] A $G$-strongly positive script is $G$-positive but the converse is
      not true. For instance, consider the graph in Figure \ref{F:critdigraph}.
      The scripts $(1,2,4)$ and $(0,0,1)$ both are $G$-positive but the former
      is $G$-strongly positive while the latter is not.
    \item[-] $G$-positive and $G$-strongly positive scripts always exist.
      Indeed, since the entries of $\Delta^{-1}$ are positive rational numbers,
      one can find a positive integer vector $u$ such that $u \Delta^{-1}$ is
      also a positive integer vector. Then $\sigma = u
      \Delta^{-1}$ is a $G$-strongly positive script.
    \item[-] Let $\sigma$ be a $G$-strongly positive script. The support of
      $\sigma\Delta$ may not intersect all strongly connected components of $G$. For example, 
      consider the graph in Figure \ref{F:critdigraph}. The script
      $\sigma = (1, 2, 3)$ is $G$-strongly positive and $(\sigma
      \Delta)(v_3) = 0$.
  \end{itemize}
\end{remarks}
\begin{figure}[hbt]
\centering
\includegraphics[scale=0.8]{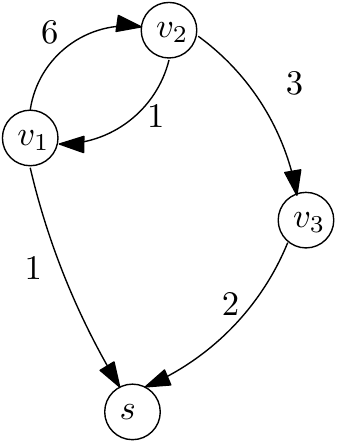}
\caption{A graph with global sink $s$ \label{F:critdigraph}}
\end{figure}

\medskip

The following lemma is a crucial result which shows that for stable configurations, legal firing consumes less time than reverse firing. 

\begin{lemma}\label{L:prec}
Let $\sigma=(\sigma_1,\dots,\sigma_n)$ be an $n$-script and let $\pmb a$ be a stable
configuration of $\CFG(G)$ such that $\pmb a +\sigma \Delta$ is non-negative. Let $\tau$ be the firing script in the stabilizing process of $\pmb a+\sigma\Delta$. 
Then $\tau\preceq \sigma$.
\end{lemma}

\begin{proof}
Let $\pmb b=\pmb a+\sigma\Delta$. Since $\pmb a$ is stable, we have, for all
$i=1,2\dots,n$:
\begin{equation}
  0\leq a_i=b_i-\sigma_i\Delta_{ii}+\sum_{j \neq i}\sigma_j\Delta_{ji}<
  \Delta_{ii}\,. \label{E:1}
\end{equation} 

Let $s=(s_1,s_2,\dots,s_k)$ be a legal firing sequence in the stabilizing process
of $\pmb b$. Note that some of $s_i$ could appear more than once. Denote by
$s^{\leq t}=(s_1, s_2,\dots,s_t)$ the legal firing sequence which is the
prefix of $s$ of length $t$ and $\tau^{\leq t}$ its associated firing script,
{\em i.e.} $\tau_j^{\le t}$ is the number of firings of the vertex $j$ until time
$t$.

Assume to the contrary that $\tau \npreceq \sigma$. Let $t_0$ be the first time
in the process of firing $s$ that the number of occurrences of the vertex
$i = s_{t_0}$ exceeds $\sigma_{s_{t_0}}$. We have:
\begin{equation*}
  \begin{cases}
    \tau^{\leq t_0}_j=\sigma_j+1 &\text{ if } j=i,\\
    \tau^{\leq t_0}_j\leq \sigma_j &\text{ if } j\ne i.
  \end{cases}
\end{equation*} 
Combining this with the fact that $\Delta_{ij} \le 0$ for all $j \neq i$ and
(\ref{E:1}), we have the following 
evaluation for the number of chips at vertex $i$ after the
firing of $s^{\le t_0}$:
\begin{align*}
  &b_i-\tau^{\leq t_0}_t\Delta_{ii}-\sum_{j\ne i} \tau^{\leq t_0}_j
  \Delta_{ji}\leq b_i-(\sigma_i+1)\Delta_{ii}-\sum_{j\ne i}
  \sigma_j\Delta_{ji} < 0\,,
\end{align*} 
which contradicts the legality of $s$.
\end{proof}

Note that the proof of Lemma \ref{L:prec} does not require $\pmb a+\sigma\Delta$
to be non-negative. We keep this hypothesis, however, for ease of exposition and
in fact, it does not affect our other proofs.

The next lemma states that reverse-firing a $G$-positive script followed by the 
stabilization does not decrease the weight of a non-negative configuration.

\begin{lemma}\label{L:gre} 
  Let $\sigma$ be a $G$-positive script and let $\pmb a,\pmb b$ be stable 
  configurations such that $\pmb b=(\pmb a+\sigma\Delta)^\circ$. Then $w(\pmb
  b)\geq w(\pmb a)$.
\end{lemma}
\begin{proof}
  Let $\tau$ be the firing script of the stabilization of
  $\pmb a+\sigma\Delta$. We have $\pmb b=\pmb a+\sigma\Delta-\tau\Delta$, and:
  \begin{align*}
    w(\pmb b)=\pmb b\cdot \pmb 1&=\pmb a \cdot \pmb 1+ (\sigma-\tau)\Delta \cdot \pmb 1\\
    & =w(\pmb a)+ \sum_{i: (i,n+1)\in E} (\sigma_i-\tau_i) e_{i,n+1}.
  \end{align*}
  By Lemma \ref{L:prec}, $\sigma_{i}\geq \tau_{i}$ for all $i$. Hence, $w(\pmb b)\geq
  w(\pmb a)$.
\end{proof}

The next proposition shows the recurring nature of critical configurations. This
result was mentioned in \cite[Theorem 2.27]{PPW11}, but we have not found a proof
in the literature. We restate it here with a complete proof.

\begin{proposition}\label{P:cri}
Let $\pmb a$ be a stable configuration and let $\sigma$ be a $G$-strongly
positive script. Then $\pmb a$ is critical if and only if $(\pmb
a+\sigma\Delta)^\circ=\pmb a$.
Moreover, if $\pmb a$ is critical then the firing script in the stabilizing
process of $\pmb a+\sigma\Delta$ is $\sigma$.
\end{proposition}
\begin{proof}
We first prove that $(\pmb c_{max}+\sigma\Delta)^\circ=\pmb c_{max}$. 
Since $\sigma\Delta \succeq 0$, $(\pmb c_{max}+\sigma\Delta)^\circ$ is stable,
hence $(\pmb c_{max}+\sigma\Delta)^\circ\preceq \pmb c_{max}$. 
On the other hand, by Lemma \ref{L:gre}, $w(\pmb
c_{max}+\sigma\Delta)^\circ\geq w(\pmb c_{max})$. Hence,
$(\pmb c_{max}+\sigma\Delta)^\circ= \pmb c_{max}$. 

Now let $\pmb a$ be an arbitrary critical configuration. By
Lemma~\ref{lem:critproperties}, there
exists $\pmb c\succeq 0$ such that $\pmb a=(\pmb c_{max}+\pmb c)^\circ$. Hence:
\begin{align*}
(\pmb a+\sigma\Delta)^\circ &=((\pmb c_{max}+\pmb c)^\circ+\sigma\Delta)^\circ\\
&=(\pmb c_{max}+\pmb c+\sigma\Delta)^\circ \\ 
&=((\pmb c_{max}+\sigma\Delta)^\circ+\pmb c)^\circ \\
&=(\pmb c_{max}+\pmb c)^\circ=\pmb a.  
\end{align*}

Conversely, suppose that $(\pmb a+\sigma\Delta)^\circ=\pmb a$. Observe that if
$u$, $v$ are vertices of $G$ and there is a directed path from $u$ to $v$, then
for all $p > 0$, by putting sufficiently many chips in $u$ and repeatedly (and
legally) firing the vertices along that path, one can increase the number of
chips in $v$ by at least $p$ (for instance, if $u = v_0, v_1, \cdots, v_k = v$
is a directed path from $u$ to $v$, then $p \prod_{i = 0}^{k - 1}
\mbox{d}^+(v_i)$ chips in $u$ are enough). Thus, for all $p > 0$, starting with
sufficiently many chips in one vertex of each source component of $G$, one can
redistribute chips in $G$ such that every non-sink vertex has at least
$p$ chips. Since $\sigma$ is $G$-strongly positive, the support of
$\sigma\Delta$ on each source component is non-empty.
Therefore, for a sufficiently large number $m \in \mathbb N$, there exists a
legal firing sequence leading from $\pmb a + m \sigma
\Delta$ to a configuration $\pmb a' \succeq \pmb c_{max}$. In other words, there
exist a positive integer $m$, a positive configuration $\pmb c$ and a legal
firing sequence from $\pmb a + m \sigma \Delta$ with firing script $\tau$ such
that:
$$\pmb c_{max}+\pmb c=(\pmb a+m\sigma\Delta)-\tau\Delta\,,$$ 
whence
\[
  (\pmb c_{max} + \pmb c)^\circ = (\pmb a + m \sigma \Delta - \tau
  \Delta)^\circ = (\pmb a + m \sigma \Delta)^\circ\,.
\]

On the other hand, since $(\pmb a+\sigma\Delta)^\circ=\pmb a$ and $\sigma\Delta
\succeq \pmb  0$, we have, by induction:
\begin{align*}
  (\pmb a + m \sigma \Delta)^\circ &= (\pmb a +\sigma\Delta+ (m - 1) \sigma \Delta)^\circ=((\pmb a+\sigma\Delta)^\circ + (m - 1) \sigma \Delta)^\circ\\
&=(\pmb a + (m - 1) \sigma \Delta)^\circ =
  \dots = (\pmb a + \sigma \Delta)^\circ = \pmb a\,.
\end{align*}

So $(\pmb c_{max} + \pmb c)^\circ = \pmb a$, which means $\pmb a$ is critical.

Finally, suppose $\pmb a$ is critical and let $\tau$ be the firing script of the
stabilization of $\pmb a + \sigma \Delta$. We have $\pmb a + \sigma \Delta -
\tau \Delta = \pmb a$, hence $(\tau - \sigma) \Delta = \pmb 0$. Since $\Delta$
is non-singular, $\tau = \sigma$.
\end{proof}

\begin{remark}
  The above proof shows that the ``only if'' part also holds when $\sigma$ is
  $G$-positive but not $G$-strongly positive. The ``if'' part, however, requires
  $G$-strong positiveness to be correct. Consider the graph in Figure
  \ref{F:critdigraph} and the scripts $\sigma_1=(0,0,1)$, $\sigma_2=(1,2,4)$. We
  can easily verify that $\sigma_1$ is $G$-positive but not $G$-strongly
  positive and $\sigma_2$ is $G$-strongly positive. Let $\pmb a=(1,1,1)$. Then $(\pmb
  a+\sigma_2\Delta)^\circ=(6,3,1)\ne \pmb a$, which implies that $\pmb a$ is not
  critical. Meanwhile, $(\pmb a+\sigma_1\Delta)^\circ=\pmb
  a$, which means that $\pmb a$ is ``recurrent'' when reverse-firing by the script
  $\sigma_1$. 
\end{remark}

\medskip

In \cite{PP16}, Perrot and Pham posed a question about the 
maximum weight of critical configurations in their equivalence classes and gave
the answer for Eulerian digraphs. Here, we give the answer for general
digraphs with a global sink.

\begin{theorem} \label{T:max} 
  Let $\pmb a$ be a critical configuration and $\pmb b$
  be a stable configuration in the equivalence class of $\pmb a$. Then $w(\pmb
  a)\geq w(\pmb b)$.
\end{theorem}

\begin{proof}
  Let $\sigma$ be a $G$-strongly positive script. Consider the following
  sequence of stable configurations:
  \[
    \pmb b_0 = \pmb b\,, \pmb b_{k + 1} = (\pmb b_k + \sigma \Delta)^\circ
    \mbox{ for all } k \ge 0\,.
  \]
  By Lemma \ref{L:gre}, $\{w(\pmb b_k)\}$ is non-decreasing. In particular,
  $w(\pmb b_k) \ge w(\pmb b)$ for all $k$.

  The stable configurations $(\pmb b_k)$ all belong to the equivalence class of
  $\pmb a$, which is finite, so there exist $k < \ell$ such that $\pmb b_k =
  \pmb b_{\ell}$, or:
  \[
   \pmb b_k= \pmb b_\ell = (\pmb b_k + (\ell - k) \sigma \Delta)^\circ\,.
  \]
  Note that if $\sigma$ is $G$-strongly positive, then $m \sigma$ is also
  $G$-strongly positive for all positive integer $m$. So applying Proposition
  \ref{P:cri} with $\pmb b_k$ and $(\ell - k) \sigma$ yields that $\pmb b_k$ is
  critical.
  But since each equivalence class has only one critical configuration, $\pmb
  b_k = \pmb a$. Thus $w(\pmb a) \ge w(\pmb b)$.
\end{proof}

As mentioned at the end of section \ref{sec:prelim},
to decide whether a configuration is critical, one may use Dhar's burning
algorithm~\cite{Dha90} for undirected graphs and Speer's script
algorithm~\cite{Spe93} for strongly connected digraphs. 
Both algorithms are based on Proposition \ref{P:cri}: the idea is to use a 
minimal $G$-strongly positive script to test for the recurrence of a critical configuration. 
In Dhar's algorithm, the minimum script is always $(1,1,\dots,1)$. 
In Speer's algorithm, the script is not explicitly given, but one can construct
it through a greedy iteration process: starting with the initial script 
$\sigma=(1,0,\dots,0)$, at each step, one chooses an index $i$ such that
$(\sigma\Delta)_i$ is negative and increments $\sigma_i$.
In 2011, Perkinson et al.~\cite{PPW11} developed
the script algorithm for digraphs with a global sink by taking the initial
script $\sigma=(1,1,\dots,1)$. This initial script indicates that every vertex
has to fire at least once in a recurrent firing sequence. 
It is showed in \cite{Spe93} that the output of the script algorithm is uniquely
determined and is in fact the unique minimum $G$-strongly positive script. We denote this script by $\sigma^M$. 

{\bf Example:} For the graph in Figure \ref{F:critdigraph}, the minimum
$G$-strongly positive script is $\sigma^M=(1,2,3)$. 

\medskip

The next theorem is an extension of Aval \emph{et al.}'s result~\cite{ADDL16} for
undirected graphs to digraphs with a global sink. Essentially, it states that
one cannot obtain a stable configuration from a critical one by reverse-firing a
non-empty multi-subset of vertices. This result is instrumental in the proof of
the energy maximizing characterization of critical configurations (Theorem
\ref{T:max2}) and in the proof of the duality theorem (Theorem \ref{T:dua}).

\begin{theorem}\label{T:script}
  Let $\sigma^M$ be the minimum $G$-strongly positive script and let $\pmb a$ be a
  stable configuration of $\CFG(G)$. The following statements are equivalent:
  \begin{enumerate}[label=\roman*)]
    \item $\pmb a$ is critical;
    \item $\pmb a+\tau\Delta$ is not stable for all $\tau \succ \pmb 0$;
    \item $\pmb a+\tau\Delta$ is not stable for all $\tau$ such that $\pmb 0 \prec
      \tau \prec \sigma^M$.
  \end{enumerate}
\end{theorem}

\begin{proof}\mbox{}

  $i)\Rightarrow ii)$:
  Let $\pmb a$ be critical and suppose that there exists a script $\tau \succ 
  \pmb 0$ such that $\pmb b = \pmb a + \tau \Delta$ is stable. Let $m =
  \max\{\tau_i\} + 1$ and $\sigma = m \sigma^M$. Then $\sigma - \tau \succ \pmb 0$.

  Since $\sigma$ is $G$-strongly positive, we have, by
  Proposition \ref{P:cri}, $(\pmb a + \sigma \Delta)^\circ = \pmb a$ and that
  $\sigma$ is the stabilizing script of $\pmb a + \sigma \Delta$ which is also equal to $\pmb b +
  (\sigma - \tau) \Delta$. But by Lemma \ref{L:prec}, the stabilizing script of
  $\pmb b + (\sigma - \tau) \Delta$ cannot exceed $\sigma - \tau$,
  contradiction!

  $ii)\Rightarrow iii)$: Trivial.

  $iii)\Rightarrow i)$: Let $\pmb b = (\pmb a + \sigma^M \Delta)^\circ$. Suppose that
  $\pmb b \neq \pmb a$. Let $\tau$ be the stabilizing script of $\pmb a +
  \sigma^M \Delta$. By Lemma \ref{L:prec} and since $\pmb b \neq \pmb a$, $\tau
  \prec \sigma^M$. Now $\pmb b = \pmb a + (\sigma^M - \tau) \Delta$ is stable
  and $\pmb 0 \prec \sigma^M - \tau \prec \sigma^M$, contradiction! So $(\pmb a +
  \sigma^M \Delta)^\circ = \pmb a$, which means $\pmb a$ is critical.
\end{proof}

\medskip

Theorem \ref{T:max} shows that a critical configuration has maximum weight among
all of the stable configurations of its equivalence class. This property,
however, does not characterize critical configurations, in the sense that there
may be non-critical configurations with the same weight, see Figure
\ref{F:sameweight} for an example. The notion of energy vector, defined
hereafter, will 
provide a maximal characterization of critical configurations.

\begin{definition}[$\CFG$ order and energy vector]\mbox{}

  Let $\pmb a$ and $\pmb b$ be two configurations. We say that $\pmb a
  \preceq_{\CFG} \pmb b$ if $\pmb a \Delta^{-1} \preceq \pmb b \Delta^{-1}$. The
  vector $\pmb a \Delta^{-1}$ is called the {\em energy vector} of the
  configuration $\pmb a$.
\end{definition}

\begin{figure}[hbt] 
  \begin{center}
    \includegraphics[scale=0.7]{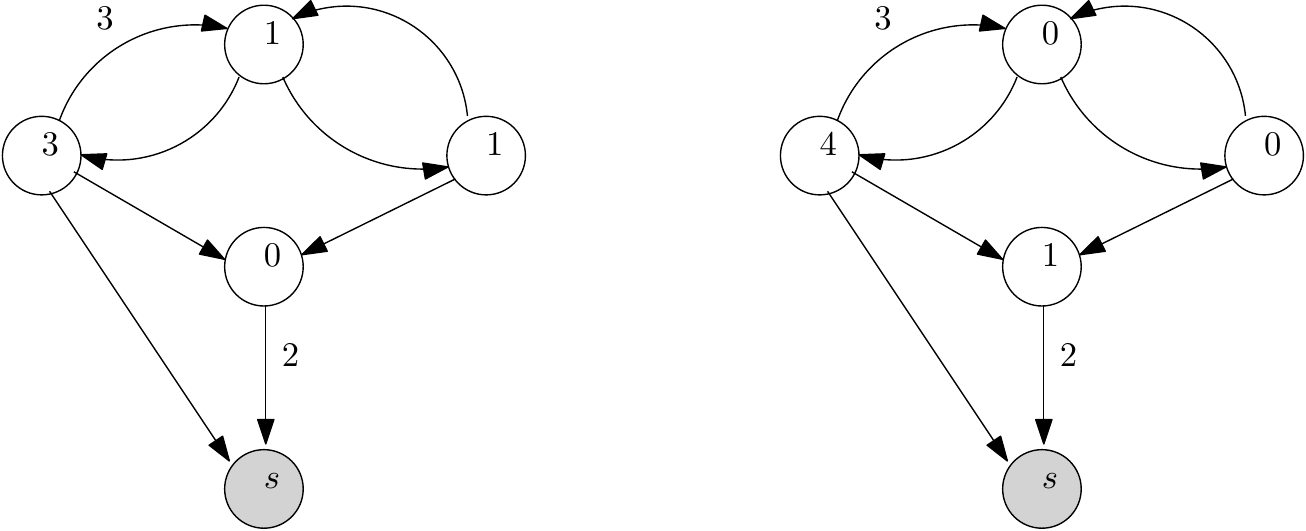}
    \caption{Two stable configurations in the same equivalence class with the
    same weight of $5$. The left one is critical whereas the right one is not. \label{F:sameweight}}
  \end{center}
  
\end{figure}

The following properties are immediate. 

\begin{lemma}\label{L:evectorproperty}\mbox{} Let $\pmb a, \pmb b$ be two configurations of $\CFG(G)$. Then   
  \begin{enumerate}[label=\roman*)]
    \item The binary relation $\preceq_{\CFG}$ is a partial order on $\mathbb Z^n$.
    \item If $\pmb a \sim \pmb b$ and $\pmb a \preceq_{\CFG} \pmb b$, then there
      exists a unique non-negative script $\sigma$ such that $\pmb b - \pmb a =
      \sigma \Delta\,.$
  \end{enumerate}
\end{lemma}

\begin{remark}
  In the chip-firing game literature, $\pmb b$ is said to be {\em accessible}
  from $\pmb a$ if there is a legal firing sequence leading from $\pmb a$ to
  $\pmb b$. The accessibility relation induces a partial order on the set of
  non-negative configurations. It is easy to see that the accessibility order is
  a sub-order of $\preceq_{\CFG}$.
\end{remark}

\begin{comment}
Now let $\pmb a, \pmb b\in \CFG(G)$. We say that $\pmb a\preceq_{\CFG} \pmb b$
if $\pmb a\Delta^{-1} \preceq \pmb b\Delta^{-1}$. The vector $pmb a \Delta^{-1}$
is called an \emph{energy vector} of $\pmb a$. It is easy to check that
$\preceq_{\CFG}$ is a partial order on $\CFG(G)$. In case of $\pmb a\sim \pmb b$
and $\pmb a\preceq_{\CFG} \pmb b$, then $\pmb b-\pmb a$ is represented uniquely
as a linear combination with non-negative coefficients of the rows of $\Delta$ .
Note that in the literature ones consider the accessibility relation of a
$\CFG$: $\pmb a$ is accessible from $\pmb b$ if $\pmb a$ is obtained from $\pmb
b$ by a legal firing sequence. The set of configurations accessible from an
initial configuration together with the accessibility relationship forms a
partial order and more precisely a lattice~\cite{LP01}. It is easily seen that
the order induced by the accessibility is in fact a suborder of the order
$\preceq_{\CFG}$. We have the following
\end{comment}

\begin{theorem}\label{T:max2} 
  In the set of stable configurations of each
  equivalence class, the critical configuration is the greatest 
  with respect to the $\preceq_{\CFG}$ order. Equivalently, the critical
  configuration is of the greatest energy (w.r.t. the containment order).
\end{theorem}

\begin{proof} Let $\pmb a$ be an arbitrary stable configuration and $\pmb a^*$
  be the critical configuration in its equivalence class.
  Assume that $\pmb a$ is not
  critical. We will prove that $\pmb a\prec_{\CFG} \pmb a^*$. 
  
  Since $\pmb a$ is not critical, by Theorem
  \ref{T:script}, there exists a script $\tau_1\succ \pmb 0$ such that $\pmb
  a_1=\pmb a+\tau_1\Delta$ is stable. Then $\pmb a_1\sim \pmb a$ and $\pmb
  a_1\succ_{\CFG} \pmb a$ (since $\pmb a_1\Delta^{-1}-\pmb
  a\Delta^{-1}=\tau_1\succ \pmb 0$).
  Now, if $\pmb a_1$ is critical, then $\pmb a_1=\pmb a^*$, hence $\pmb
  a^*=\pmb a_1\succ_{\CFG} \pmb a$. If $\pmb a_1$ is not critical, by Theorem
  \ref{T:script}, there exists a script $\tau_2 \succ \pmb 0$ such
  that $\pmb a_2=\pmb a_1+\tau_2\Delta$ is stable, and so $\pmb a\prec_{\CFG} \pmb
  a_1\prec_{\CFG} \pmb a_2$. Repeating the above arguments, we obtain a
  sequence of stable configurations $\pmb a\prec_{\CFG} \pmb a_1\prec_{\CFG} \pmb
  a_2\prec_{\CFG}\cdots$ 
  Since the number of stable configurations in each equivalence class is finite,
  this increasing sequence (w.r.t. the $\preceq_{\CFG}$ order) must stop at a
  critical configuration which in fact coincides with $\pmb a^*$. Thus, $\pmb a^*\succ_{\CFG}
  \pmb a$. 
\end{proof}

\section{The duality between critical configurations and superstable
configurations}\label{sec:duality} 
In this section, we revisit the duality
between critical configurations and superstable configurations (Theorem
\ref{T:dua}). It has been shown, by Asadi and Backman~\cite[Theorem 3.15]{AB11} and Perkinson
\emph{et al.}~\cite[Corollary 5.15]{PPW11} that for a digraph with a global sink,
a configuration is superstable if and only if its complementary configuration is
critical.  The proof of Perkinson \emph{et al.} used advanced algebra techniques
such as the coordinate ring induced by the Laplacian matrix and Gr\"obner bases,
while Asadi and Backman's proof is combinatorial. We give another
combinatorial proof which is different and independent from Asadi and Backman's one.

\begin{definition}[Super-stable configuration] \mbox{} 

  A non-negative configuration $\pmb a$ is \emph{superstable} if for all scripts
  $\sigma \succ \pmb 0$, $\pmb a - \sigma\Delta$ has a negative component.
\end{definition} 

Super-stable configurations are also known as \emph{reduced divisors} or
\emph{$z$-superstable configurations}~\cite{BS13, GK15}.  It can be readily seen
that if $\pmb a$ is superstable and 
$\pmb 0 \preceq \pmb b\preceq \pmb a$, then $\pmb b$ is also superstable.
Moreover, superstable configurations are stable. Indeed, if $\pmb a$ is not
stable and $\tau$ is the stablizing script of $\pmb a$, then $\tau \succ \pmb 0$
and $\pmb a-\tau\Delta = \pmb a^\circ$ is non-negative.

\medskip

We now state the duality theorem.

\begin{theorem}[Duality Theorem]\label{T:dua} Let $\pmb a$ be a stable
  configuration. Then $\pmb a$ is critical if and only if $\pmb c_{max}-\pmb a$
  is superstable. Consequently, the number of critical configurations is equal
  to the number of superstable configurations.  
\end{theorem} 

\begin{proof} 
  Let $\pmb a$ be a critical configuration. Suppose that $\pmb
  c_{max} - \pmb a$ is not superstable. Then there exists a script $\sigma \succ
  \pmb 0$ such that $(\pmb c_{max} - \pmb a) - \sigma \Delta \succeq \pmb 0$.
  Let $\pmb b = (\pmb c_{max} - \pmb a - \sigma \Delta)^\circ$ and let $\tau$ be
  the stabilizing script of $\pmb c_{max} - \pmb a - \sigma \Delta$. Then $\pmb
  a + (\sigma + \tau) \Delta = \pmb c_{max} - \pmb b$, which contradicts Theorem
  \ref{T:script} since $\sigma + \tau \succ \pmb 0$ and $\pmb c_{max} - \pmb b$
  is stable.  So $\pmb c_{max} - \pmb a$ is superstable.

  \smallskip

  Conversely, suppose that $\pmb c_{max} - \pmb a$ is superstable for some
  non-critical stable configuration $\pmb a$. By Theorem \ref{T:script}, there
  exists some script $\sigma \succ \pmb 0$ such that $\pmb a + \sigma \Delta$ is
  stable. In particular, $\pmb a + \sigma \Delta \preceq \pmb c_{max}$, hence
  $(\pmb c_{max} - \pmb a) - \sigma \Delta \succeq \pmb 0$, contradiction!
\end{proof} 

\begin{corollary}\label{C:dua} 
  Let $\pmb a$ be a stable configuration and $\sigma^M$ be the minimum
  $G$-strongly positive script. The following statements are equivalent:
  \begin{enumerate}[label=\roman*)] 
    \item $\pmb a$ is superstable; 
    \item $\pmb a-\sigma\Delta$ has a negative component for all $\pmb
      0 \prec \sigma \preceq \sigma^M$;
    \item $\pmb a-\sigma\Delta$ is not stable for all $\pmb 0\prec \sigma \preceq
      \sigma^M$.
  \end{enumerate} 
\end{corollary}

\smallskip

Recall that for undirected graphs, $\sigma^M=(1,1,\dots,1)$. By Corollary
\ref{C:dua}, in order to check if a configuration $\pmb a$ is superstable, it
suffices to check the non-negativeness or non-stableness of configurations
obtained from $\pmb a$ by firing a \emph{subset} of vertices. This is not true
for digraphs in general. Consider, for instance,  the digraph in Figure
\ref{F:Fig-sta-3}: the configuration $\pmb a=(0,3)$ is not superstable ($\pmb a
- (2, 1) \Delta = (1, 1) \succ \pmb 0$) despite the fact that $\pmb a - \sigma \Delta$ has 
at least one negative component for all $\sigma \in \left\{ (0, 1)\,, (1, 0)\,,
(1, 1) \right\}$.

\medskip

The energy minimizing characteristic of superstable configurations 
is a natural consequence of the duality and the energy
maximizing characteristic of critical configurations.

\begin{corollary}\label{C:sma} 
  Among the non-negative configurations of each
  equivalence class, the superstable configuration is the smallest with respect
  to the $\preceq_{\CFG}$ order.  
\end{corollary} 

\begin{proof} 
  The result follows Theorem \ref{T:max2}, Theorem \ref{T:dua}, and the following
  observations:
  \begin{enumerate}[label=\roman*)]
    \item Since $\pmb a \succeq_{\CFG} \pmb a^\circ$ for all non-negative
      configuration $\pmb a$, it suffices to show that superstable
      configurations are the smallest among the stable configurations of each
      equivalence class.
    \item If $\pmb a$ and $\pmb b$ are equivalent stable configurations, then
      $\pmb c_{max} - \pmb a$ and $\pmb c_{max} - \pmb b$ are also equivalent
      stable configurations.
    \item For any configurations $\pmb a$ and $\pmb b$, $\pmb a
      \succeq_{\CFG} \pmb b$ if and only if $\pmb c_{max} - \pmb a
      \preceq_{\CFG} \pmb c_{max} - \pmb b$.
  \end{enumerate}
\end{proof}

\medskip 

\begin{figure}[hbt]
  \begin{center}
    \includegraphics[scale=0.7]{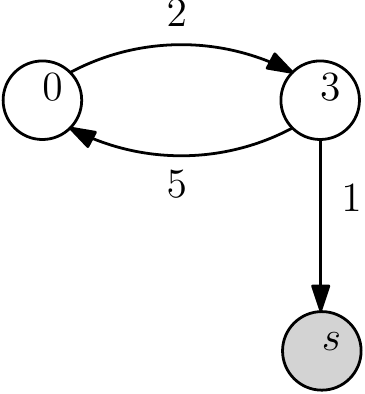} 
    \caption{A non-superstable
    configuration which is non-negative when firing any subsets of vertices.}
    \label{F:Fig-sta-3} 
  \end{center} 
\end{figure} 

We conclude with a conjecture on the $\preceq_{\CFG}$ order. 
Let $\pmb a_*$ be a superstable configuration and let $\pmb a^*$ be the critical
configuration equivalent to $\pmb a_*$. Consider the following increasing
sequence of stable configurations with respect to the $\preceq_{CFG}$ order:
\begin{equation}
  \pmb a_* \preceq_{CFG} (\pmb a_* + \sigma^M \Delta)^\circ \preceq_{CFG} (\pmb
  a_* + \sigma^M \Delta)^\circ \preceq_{\CFG} \cdots \preceq_{\CFG} \pmb a^*
  \label{eq:linseq}
\end{equation}

Consider the graph $G$ shown in Figure \ref{F:Fig-crit-2} with $\pmb a_* =
(1, 0, 0, 1))$ and $\pmb a^* = (3, 1, 1, 0)$. For this graph:
\[
  \Delta = \begin{bmatrix}
    5 & -3 & 0 & -1 \\
    -1 & 2 & -1 & 0 \\
    0 & -1 & 2 & -1 \\
    0 & 0 & 0 & 2
  \end{bmatrix}\,, \sigma^M = (1, 2, 1, 1)\,,
\]
and we have the sequence:
\[
  (1, 0, 0, 1) \prec_{\CFG} (0, 1, 1, 0) \prec_{\CFG} (4, 0, 0, 1)
  \prec_{\CFG} (3, 1, 1, 0)\,.
\]
The above sequence contains all the stable configurations equivalent to
$(1, 0, 0, 1)$. In fact, we have not found an example in which the sequence
(\ref{eq:linseq}) does not cover all the stable configurations of the
corresponding equivalence class. We conjecture that this is true for digraphs
with a global sink in general.

\begin{figure}[hbt]
  \begin{center}
    \includegraphics[scale=0.7]{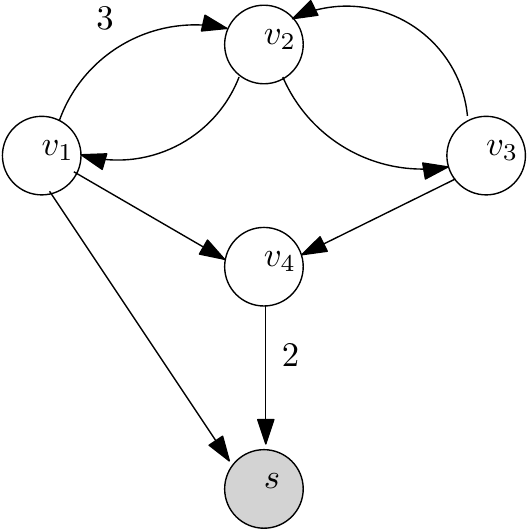} 
    \caption{Example for the Conjecture} 
    \label{F:Fig-crit-2} 
  \end{center}
\end{figure}

\begin{conjecture} 
  The $\preceq_{CFG}$ order is linear on the set of all stable configurations of
  each equivalence class.
\end{conjecture}

\section*{Acknowledgment}\label{sec:ack}
\addcontentsline{toc}{section}{Acknowledgements}
The manuscript was conceptualized when the authors
visited the Vietnam Institute for Advanced Study in Mathematics. We would like
to thank Prof. Robert Cori and Prof. Spencer Backman for their very useful
comments and suggestions on a  short version of this manuscript \cite{Tra16}.
The work is  partially funded by the PSSG no.6 and the NAFOSTED grant under the
project number 101.99-2016.16. 

\bibliographystyle{plain}

\begin{thebibliography}{10}

\bibitem{AB11}
A.~Asadi and S.~Backman.
\newblock Chip-firing and riemann-roch theory for directed graphs.
\newblock {\em Electronic Notes in Discrete Mathematics}, 38:63 -- 68, 2011.

\bibitem{ADDL16}
J-C. Aval, M.~D'Adderio, M.~Dukes, and Y.~{Le Borgne}.
\newblock Two operators on sandpile configurations, the sandpile model on the
  complete bipartite graph, and a cyclic lemma.
\newblock {\em Advances in Applied Mathematics}, 73:59--98, 2016.

\bibitem{BS13}
M.~Baker and R.~Shokrieh.
\newblock Chip-firing games, potential theory on graphs, and spanning trees.
\newblock {\em J. Combin. Theory Ser. A}, 120(1):164--182, 2013.

\bibitem{BN07}
M.~Baker and S.~Norine.
\newblock Riemann-roch and abel-jacobi theory on a finite graph.
\newblock {\em Adv. Math}, 215:766--788, 2007.

\bibitem{Big99}
N.~Biggs.
\newblock Chip firing and the critical group on a graph.
\newblock {\em Journal of Algebraic Combinatorics}, 9:25--42, 1999.

\bibitem{BL92}
A.~Bj\"orner and L.~Lov\'asz.
\newblock Chip firing games on directed graphes.
\newblock {\em Journal of Algebraic Combinatorics}, 1:305--328, 1992.

\bibitem{BLS91}
A.~Bj\"orner, L.~Lov\'asz, and W.~Shor.
\newblock Chip-firing games on graphes.
\newblock {\em E.J. Combinatorics}, 12:283--291, 1991.

\bibitem{CPT13}
R.~Cori, T.~H.~D.~Phan, and T.~T.~H.~Tran.
\newblock Signed chip firing games and symmetric sandpile models on the cycles.
\newblock {\em RAIRO - Theoretical Informatics and Applications},
  47(2):133--146, 2013.

\bibitem{Dha90}
D.~Dhar.
\newblock Self-organized critical state of sandpile automaton models.
\newblock {\em Phys. rev. Lett.}, 64(14):1613--1616, 1990.

\bibitem{GK15}
J.~Guzm{\'a}n and C.~Klivans.
\newblock Chip-firing and energy minimization on m-matrices.
\newblock {\em Journal of Combinatorial Theory, Series A}, 132:14 -- 31, 2015.

\bibitem{HLMPPW08}
A.~E.~Holroyd, L.~Levine, K.~M{\'e}sz{\'a}ros, Y.~Peres, J.~Propp, and
D.~B.~Wilson.
\newblock Chip-firing and rotor-routing on directed graphs.
\newblock In {\em In and out of equilibrium. 2}, volume~60 of {\em Progr.
  Probab.}, pages 331--364. Birkh\"auser, Basel, 2008.

\bibitem{LP01}
M.~Latapy and H.D.~Phan.
\newblock The lattice structure of chip firing games.
\newblock {\em Physica D}, 115:69--82, 2001.

\bibitem{Mer05}
C.~Merino.
\newblock The chip-firing game.
\newblock {\em Discrete Math.}, 302(1-3):188--210, 2005.

\bibitem{Mer97}
C.~Merino.
\newblock Chip firing and the {T}utte polynomial.
\newblock {\em Ann. Comb.}, 1(3):253--259, 1997.

\bibitem{PPW11}
D.~{Perkinson}, J.~{Perlman}, and J.~{Wilmes}.
\newblock {Primer for the algebraic geometry of sandpiles}.
\newblock {\em ArXiv e-prints}, dec 2011.

\bibitem{PP16}
K.~Perrot and V.~T.~Pham.
\newblock Chip-firing game and a partial {T}utte polynomial for eulerian
  digraphs.
\newblock {\em Electr. J. Comb.}, 23(1):P1.57, 2016.

\bibitem{plemmons77}
R.J.~Plemmons.
\newblock M-matrix characterizations. i - nonsingular m-matrices.
\newblock {\em Linear Algebra and its Applications}, 18(2):175 -- 188, 1977.

\bibitem{PS04}
A.~Postnikov and B.~Shapiro.
\newblock Trees, parking functions, syzygies, and deformations of monomial
  ideals.
\newblock {\em Trans. Amer. Math. Soc.}, 356:3109--3142, 2004.

\bibitem{Spe93}
Eugene~R.~Speer.
\newblock Asymmetric abelian sandpile models.
\newblock {\em Journal of Statistical Physics}, 71(1-2):61--74, 1993.

\bibitem{Tra16}
T.~T.~H.~Tran.
\newblock G-strongly positive scripts and critical configurations of chip
  firing games on digraphs.
\newblock In {\em 10th International Conference on Advanced Computing and
  Applications, Can Tho City, Vietnam, November 23-25}, pages 151--157, 2016.

\end{thebibliography}

\addcontentsline{toc}{section}{References}

\end{document}